\documentclass[12pt]{article}
\usepackage{amsmath}
\usepackage{amssymb}
\usepackage{bbm}
\usepackage[margin=1.25cm]{geometry}
\usepackage{hyperref}
\usepackage{tikz}
\usepackage{setspace}

\usetikzlibrary { decorations.pathmorphing, decorations.pathreplacing, decorations.shapes, }

\newtheorem{thm}{Theorem}[section]

\newtheorem{lemma}[thm]{Lemma}

\newtheorem{theorem}[thm]{Theorem}
\newtheorem{remark}[thm]{Remark}

\newtheorem{definition}[thm]{Definition}

\newtheorem{conj}[thm]{Conjecture}

\newenvironment{proof}{{\bf Proof:}}{\hfill$\square$\vskip.5cm}

\newcommand{\N}{\mathbb{N}}

\begin{document} 

\title{About the Cram\'er Large Deviation Property for Bell Polynomials}
\author{Sophia Li${}^1$ and Shannon Starr${}^2$\\
\small
${}^1$ Vestavia Hills High School, 2235 Lime Rock Rd, Vestavia Hills, AL 35216\\
\small
${}^2$ University of Alabama at Birmingham, Math.~Dept., 1402 Tenth Avenue South, Birmingham, AL 35294--1241}
\date{9 September 2026}

\maketitle

\abstract{
If $\boldsymbol{w} = (w_1,w_2,\dots)$ is a sequence in $\N=\{1,2,\dots\}$, the partial Bell polynomials based on $\boldsymbol{w}$ are
$B_{n,k}$ for $k \in \N$ and $n\in\{k,k+1,\dots\}$. 
Let $F(z) = \sum_{n=1}^{\infty} (w_n/n!)z^n$ be the exponential generating function for $\boldsymbol{w}$,
and assume the radius of convegence is positive $R>0$. 
Then $F(z)^k = \sum_{n=k}^{\infty} (k!/n!) z^n B_{n,k}$ for $|z|<R$.
Alternatively, defining $Q_{k,n} = (k!/n!)B_{n,k}$, we have $Q_{1,n} = w_n/n!$, 
and  $Q_{k+1,n}=\sum_{m=1}^{n-k} Q_{1,m} Q_{k,n-m}$ for $k\geq 1$. 
Let us say that the Cram\'er-type large deviation property holds if
$$
\lim_{\substack{n \to \infty\\ k/n \to \kappa}} \frac{1}{n}\, 
\ln\left(Q_{k,n}\right)\, =\, \mathcal{G}(\kappa)\, ,$$
for every $\kappa \in (0,1)$,
where $\mathcal{G}(\kappa)=\inf_{r \in (0,R)} (\kappa \ln(F(r))-\ln(r))$.
The (Hardy-Ramanujan) Erd\"os induction argument suggests this should generally be true
as long as two technical conditions are true: one an initial step, and the other a condition for small densities $\kappa$.
}

\section{The Bell polynomials}

Suppose that $\boldsymbol{w} = (w_1,w_2,\dots)$ is a combinatorial sequence, so that in particular  $w_n$
is in $\N = \{1,2,\dots\}$ for $n \in \N$.
Then the partial Bell polynomials are defined as 
$$
B_{n,k}\, =\, \frac{n!}{k!}\, \sum_{(\nu(1),\dots,\nu(k)) \in \N^k} 
\mathbf{1}_{\{n\}}\big(\nu(1)+\dots+\nu(k)\big) \prod_{j=1}^{k} \frac{w_{\nu(j)}}{\nu(j)!}\, ,
$$
for $k \in \N$ and $n \in \{k,k+1,\dots\}$.
They are polynomials in the components of $\boldsymbol{w}$.
But we will suppress the functional dependence on those variables for notational simplicity.

We note that the output is also a positive integer  because
$$
B_{n,k}\, =\, \sum_{\{A_1,\dots,A_k\} \in \mathcal{P}(n,k)}\, \prod_{j=1}^{k} w_{|A_j|}\, ,
$$
where $\mathcal{P}(n,k)$ is the set of set partitions of $[n] = \{1,\dots,n\}$ into $k$
nonempty parts.

An excellent reference is Chapter 1 of Pitman's monograph on Combinatorial Stochastic Processes
\cite{Pitman}.
Pitman himself suggests Comtet's monograph \cite{Comtet}, who gave many references that were already present at
the time of his publication in 1974.
\section{The Cram\'er large deviation property}

Let us denote
$$
Q_{k,n}\, =\, \frac{k!}{n!}\,  B_{n,k}\, .
$$
Let us denote the generating function of the first row
$$
F(z)\, =\, \sum_{n=1}^{\infty} \frac{w_n}{n!}\, z^n\, =\, \sum_{n=1}^{\infty}  z^n Q_{1,n}\, .
$$
Let $R$ be defined as
$$
R\, =\, \sup\left(\left\{r\geq 0\, :\, F(r)<\infty\right\}\right)\, ,
$$
the radius of convergence.
Here and throughout we will assume $R>0$.
Then for $|z|<R$,
$$
F(z)^k\, =\, \sum_{n=k}^\infty Q_{k,n} z^n\, .
$$
Hence, by a positivity similar to Chernov's bound, we know for $0<r<R$ we have
$$
Q_{k,n}\, \leq\, \frac{F(r)^k}{r^n}\, .
$$
In other words,
defining
$$
\Lambda_{k,n}(r)\, =\, k\, \ln(F(r)) - n \ln(r)\, ,
$$
we have
\begin{equation}
\label{ineq:Upper}
\ln\left(Q_{k,n}\right)\, \leq\, \inf_{r \in (0,R)} \Lambda_{k,n}(r)\, .
\end{equation}

In the simplest cases of large deviation theory, we find that the elementary upper
bound actually coincides with the true large deviation behavior.
That is what we call the Cram\'er large deviation behavior.

Let us define $\mathcal{L}(r,\kappa)$ to be
$$
\mathcal{L}(r,\kappa)\, =\, \kappa \ln(F(r)) - \ln(r)\, ,
$$
so that $\Lambda_{k,n}(r) = n \mathcal{L}(r,k/n)$.

\begin{definition}
\label{def:MainOne}
We say that the $Q_{k,n}$'s satisfy the Cram\'er large deviation property if the following holds.
For each $\kappa \in (0,1)$ and any sequence $k_n$ satisfying $\lim_{n \to \infty} k_n/n = \kappa$,
$$
\lim_{n \to \infty} \frac{1}{n}\, \ln\left(Q_{k_n,n}\right)\,
=\, \mathcal{G}(\kappa)\, ,
$$
where
$$
\mathcal{G}(\kappa)\, =\, \min_{r \in [0,R]} \mathcal{L}(r,\kappa)\, .
$$
\end{definition}
In the next section, we will show that the minimum is attained.

\section{Elementary analytic properties of $\mathcal{G}$}

We need a number of specific calculations around $\mathcal{G}$.
Firstly, for each $r \in (0,R)$, let us define the probability mass function
$\boldsymbol{\mu}(r) = (\mu_1(r),\mu_2(r),\dots)$ such that
$$
\mu_n(r)\, =\, \frac{Q_{1,n} r^n}{\sum_{m=1}^{\infty} Q_{1,m} r^m}\, =\, \frac{Q_{1,n} r^n}{F(r)}\, .
$$
Then, defining
$$
\mathcal{M}(r)\, =\, \frac{r F'(r)}{F(r)}\, ,
$$
we have
$$
\mathcal{M}(r)\, =\, \sum_{n=1}^{\infty} n \mu_n(r)\, .
$$
Obviously $\mathcal{M}(r)>0$ so that $F$ is increasing on $(0,R)$.
Next, let 
$$
\mathcal{V}(r)\, =\, r\, \frac{d}{dr}\, \mathcal{M}(r)\, ,
$$
which is 
$$
\mathcal{V}(r)\, =\, \sum_{n=1}^{\infty} n^2 \mu_n(r) - \left(\sum_{n=1}^{\infty} n \mu_n(r)\right)^2\, .
$$
This is the variance of a non-degenerate distribution, so by the usual Cauchy-Schwarz inequality
$\mathcal{V}(r)>0$.
Thus, $\mathcal{M}$ is increasing on $(0,R)$.

Note that
$$
\frac{\partial}{\partial r}\, \mathcal{L}(r,\kappa)\, =\, \kappa\, \frac{F'(r)}{F(r)} - \frac{1}{r}\, .
$$
So this can be rewritten as 
$$
\frac{\partial}{\partial r}\, \mathcal{L}(r,\kappa)\, =\, \frac{1}{r}\, \left(\kappa \mathcal{M}(r) - 1\right)\, .
$$
Since $\mathcal{M}$ is increasing we have the following. If there is some $r=\rho(\kappa)$ 
such that $\mathcal{M}(r) = 1/\kappa$, then: $\frac{\partial}{\partial r} \mathcal{L}(r,\kappa)<0$ if $r<\rho(\kappa)$
and $\frac{\partial}{\partial r} \mathcal{L}(r,\kappa)>0$ if $r>\rho(\kappa)$.
Therefore, in this case $r=\rho(\kappa)$ is the unique minimizer of $\mathcal{L}(\cdot,\kappa)$.

Since the limit of the probability mass function is 
$$
\lim_{r \to 0^+} \mu_n(r)\, =\, \delta_{n,1}\, ,
$$
we see that $\rho(1) = \lim_{\kappa \to 1^-} \rho(\kappa)=0$.

Note that $\frac{d}{dt} \ln(F(e^t)) = \mathcal{M}(e^t)$ and $\frac{d^2}{dt^2} \ln(F(e^t)) = \mathcal{V}(e^t)$.
Therefore, $t\mapsto \ln(F(e^t))$ is increasing and convex.
\begin{lemma}
If $R=\infty$ then $\lim_{r \to \infty} \mathcal{M}(r)=\infty$.
\end{lemma}
\begin{proof}
By Chebyshev's inequality, for $0<t<T$ we have
$$
\sum_{n=1}^{N} \mu_n(e^T)\, \leq\, \sum_{n=1}^{N} e^{(N-n)(T-t)} \mu_{n}(e^T)\,
\leq\, e^{N(T-t)}\, \cdot 
\frac{\sum_{n=1}^{N} e^{nt} Q_{1,n}}{\sum_{n=1}^{\infty} e^{nT} Q_{1,n}}\, ,
$$
which is bounded by $\exp\left(N(T-t)+\ln(F(e^t))-\ln(F(e^T))\right)$.
But $F(e^T)>Q_{1,N+1} e^{(N+1)T}$. So $\ln(F(e^T))>(N+1)T+\ln(Q_{1,N+1})$.
Note that we have assumed that $w_{N+1}\geq 1$ so that $\ln(Q_{1,N+1})=\ln(w_{N+1}/(N+1)!)$ is not $-\infty$.
This implies
$$
\lim_{T \to \infty} \sum_{n=1}^{N} \mu_n(e^T)\, =\, 0\, .
$$
So $\lim_{r \to \infty} \mathcal{M}(r)>N+1$. But since $N$ is arbitrary, this proves the lemma.
\end{proof}
The same proof implies that if $R<\infty$ and if $F(R) = \lim_{r \to R^-}F(r)$ equals $\infty$,
then $\mathcal{M}(R)=\lim_{r \to R^-} \mathcal{M}(r)$ equals $\infty$.

In case $R<\infty$ and $F(R)<\infty$, let us define
$$
\mathcal{M}(R)\, =\, \lim_{r \to R^-} \mathcal{M}(r)\, =\, \sum_{n=1}^{\infty} n \mu_n(R)\, ,
$$
where $\mu_n(R)=Q_{1,n} R^n/F(R)$.
It is possible that $F(R)<\infty$ but $\mathcal{M}(R)=\infty$.
If $F(R)=\infty$ then we define $\mathcal{M}(R)$ to be $\infty$.

\begin{definition}
We say that the ``analytic regime'' for $\kappa \in (0,1)$ is the regime of those values satisfying
$$
\frac{1}{\kappa}\, <\, \mathcal{M}(R)\, .
$$
We say that the ``condensation regime'' for $\kappa \in (0,1)$ is the complementary regime, 
which might be empty.
\end{definition}

By the previous argument we know that if $\kappa$ is in the analytic regime then 
$\rho(\kappa) \in (0,R)$ and
$$
\mathcal{G}(\kappa)\, =\, \mathcal{L}(\rho(\kappa),\kappa)\, .
$$
But similarly, if $\kappa$ is in the condensation regime then $\rho(\kappa)=R$ and 
$$
\mathcal{G}(\kappa)\, =\, \mathcal{L}(R,\kappa)\, .
$$
Of course the condensation regime is non-empty only if $R<\infty$, $F(R)<\infty$
and $\mathcal{M}(R)<\infty$.

\begin{remark}
An elementary observation is that $\mathcal{G}$ is concave, as it is the infimum of 
affine functions (hence concave) when viewed as functions of $\kappa$.
In the condensed phase $\mathcal{G}(\kappa) = \kappa \ln(F(R)) - \ln(R)$ so it is affine.
\end{remark}

\section{The  induction argument: heuristic outline}

Let us now consider an induction argument to demonstrate why we believe the Cram\'er property,
Definition \ref{def:MainOne}, holds generally.
By (\ref{ineq:Upper}), we know
$$
\ln(Q_{k,n})\, \leq\, \inf_{r \in (0,R)} \Lambda_{k,n}(r)\, =\, n \inf_{r \in (0,R)} \mathcal{L}(r,k/n)\, ,
$$
just because $\Lambda_{k,n} = n \mathcal{L}(\cdot,k/n)$.
Hence, we already know
$$
\limsup_{n \to \infty} \frac{1}{n}\, \ln(Q_{k_n,n})\, \leq\, \mathcal{G}(\kappa)\, ,
$$
as long as $k_n$ is a sequence such that $\lim_{n \to \infty} k_n/n = \kappa$
(in part by the continuity of $\mathcal{G}$).
Now we need a matching lower bound.

In \cite{Erdos}, Paul Erd\H{o}s gave an elementary argument for the Hardy-Ramanujan
asymptotics of the partition number $p(n) = |\operatorname{Par}(n)|$
beyond the large deviation rate function
(where $\operatorname{Par}(n)$ can be thought of as 
the set of sequences $(\nu(1),\nu(2),\dots) \in \{0,1,\dots\}^{\N}$ satisfying the condition $\sum_{k=1}^{\infty} k \nu(k)=n$).
We will mimic that.
We are primarily interested in the introductory bound of that article, which is focused
on just the large deviation behavior (not the Bahadur-Rao type algebraic prefactor).
Erd\"os used the known recurrence relation for the partition numbers
$$
n p(n)\, =\, \sum_{\nu \in \N} \sum_{k \in \N} \nu p(n-k\nu)\, ,
$$
where $p(0)=1$ and $p(-m)=0$ for all $m \in \N$.
He effectively related this to the approximation of $p(n)$ at the level of large deviations $\exp(\pi \sqrt{2n/3})$, because
$$
\sum_{\nu \in \N} \sum_{k \in \N} \nu \exp\left(\pi\, \sqrt{\frac{2(n-k\nu)}{3}}\right)\,
\sim\, \exp(\pi \sqrt{2n/3}) \sum_{k\in \N} \frac{6n}{\pi^2 k}\, =\, n\exp(\pi \sqrt{2n/3})\, ,
$$
for large $n$.
Erd\"{o}s noted that Hardy and Ramanujan
themselves said that the recurrence relation should easily lead to the large deviation
rate function.

For us, the recurrence relation we will use for Bell polynomials is 
\begin{equation}
\label{eq:recurrenceHere}
Q_{k+1,n}\, =\, \sum_{m=1}^{n-k} Q_{k,n-m} Q_{1,m} \, .
\end{equation}
And we will try to relate it to the formulas we have for $r \in (0,R)$
$$
F(r)\, =\, \sum_{m=1}^{\infty} r^m Q_{1,m}\, ,
$$
for $r = \rho(\kappa)$ if we take $k_n$ to be a sequence such that 
$\lim_{n \to \infty} k_n/n = \kappa$.

\begin{definition}
For each $k \in \N$ and $n \in \{k,k+1,\dot\}$, we define the approximating sequence
$$
\Theta_{k,n}\, =\, \exp(\Lambda_{k,n}(\rho(k/n)))\, =\, \frac{F(\rho(k/n))^k}{\rho(k/n)^n}\, .
$$
\end{definition}

We wish to establish the desideratum
\begin{equation}
\liminf_{n \to \infty} \frac{1}{n}\, \ln\left(Q_{k_n,n}\right)\, \geq\, \mathcal{G}(\kappa)\, ,
\end{equation}
whenever $k_n$ is a sequence such that $\lim_{n \to \infty} k_n/n = \kappa$.
Let us define the following.
\begin{definition}
For each $\epsilon>0$, and $k \in \N$, let us define $\delta_k(\epsilon)$ such that
$$
\delta_k(\epsilon)\, =\, \inf_{n\geq k} \frac{Q_{k,n}}{\Theta_{k,n}}\, e^{\epsilon n }\, .
$$
\end{definition}
Let
$$
\delta(\epsilon)\, =\, \inf_{k \in \N} \delta_k(\epsilon)\, .
$$
If $\delta(\epsilon)>0$ then we have
$$
\liminf_{n \to \infty} \frac{1}{n}\, \ln\left(Q_{k_n,n}\right)\, \geq\, \mathcal{G}(\kappa)-\epsilon\, ,
$$
whenever $k_n$ is a sequence such that $\lim_{n \to \infty} k_n/n = \kappa$.
Therefore, one approach to proving Cram\'er's property is to establish that $\delta(\epsilon)>0$.

\section{Taylor expansion for the ratio}

We wish to show that if $\delta_k(\epsilon)>0$ then $\delta_{k+1}(\epsilon)>0$.
We also wish to show that if $\delta_k(\epsilon)>0$ for sufficiently large $k$ (depending on $\epsilon$) then $\delta_{k+1}(\epsilon)\geq \delta_k(\epsilon)$.

Let us assume that $\delta_k(\epsilon)>0$. Then
\begin{equation}
\begin{split}
\frac{Q_{k+1,n}}{\Theta_{k+1,n}}\, e^{\epsilon n}\, 
&=\, \frac{1}{\Theta_{k+1,n}}\, e^{\epsilon n }
\sum_{m=1}^{n-k} Q_{1,m} Q_{k,n-m}\\ 
&\geq\, \delta_k(\epsilon) \sum_{m=1}^{n-k} Q_{1,m} e^{\epsilon m }
\frac{\Theta_{k,n-m}}{\Theta_{k+1,n}}\, .
\end{split}
\end{equation}
Let us write
$$
\ln\left(\frac{\Theta_{k,n-m}}{\Theta_{k+1,n}}\right)\, =\, 
\Lambda_{k,n-k}\left(\rho\left(\frac{k}{n-m}\right)\right) - \Lambda_{k+1,n}\left(\rho\left(\frac{k+1}{n}\right)\right)\, .
$$
Also, note that
$$
\Lambda_{k+1,n}(r)\, =\, (k+1) \ln(F(r)) - n \ln(r)\, =\, \Lambda_{k,n-m}(r) + \ln(F(r)) -m \ln(r)\, .
$$
So we can write
\begin{equation}
\label{eq:FiniteDifference}
\frac{Q_{k+1,n}}{\Theta_{k+1,n}}\, e^{\epsilon n}\, 
\geq\, \frac{\delta_k(\epsilon)\, 
\sum_{m=1}^{n-k} Q_{1,m} \left(\rho\left(\frac{k+1}{n}\right)\right)^m e^{\epsilon m}
\exp\left(\left[ -\Lambda_{k,n-m}\circ \rho(\kappa)\right]\Bigg|_{k/(n-m)}^{(k+1)/n}\right)}{F\left(\rho\left(\frac{k+1}{n}\right)\right)}
\end{equation}
Now let us Taylor-expand the finite difference.
Let us define
$$
\Delta \kappa_{k,n,m}\, =\, \frac{k+1}{n} - \frac{k}{n-m}\, =\, \frac{n-m(k+1)}{n(n-m)}\, .
$$
Then, we note that we defined $\rho(\kappa)$ such that $\rho(k_1/n_1)$ optimizes $\Lambda_{k_1,n_1}$ for each $k_1$
and $n_1$.
So if we are in the analytic regime, then 
$$
\frac{\partial}{\partial r}\, \Lambda_{k_1,n_1}(r)\Bigg|_{r=\rho(k_1/n_1)}\, =\, 0\, .
$$
So, assuming $k/(n-m)$ is in the analytic regime, we Taylor-expand to second order.
Then we have
$$
-\Lambda_{k,n-m}\circ \rho(\kappa)\Bigg|_{k/(n-m)}^{(k+1)/n}\, \approx\, -\frac{1}{2}\, \frac{\partial^2 \Lambda_{k,n-m}(r)}{\partial r}\Bigg|_{r=\rho(k/(n-m))}
\left(\rho'\left(\frac{k}{n-m}\right)\right)^2\, \left(\Delta \kappa_{k,n,m}\right)^2\, .
$$

\section{``Uniformly Strict Cauchy-Schwarz'' property}

\begin{lemma}
We have for $k_1/n_1$ in the ``analytic regime,''
$$
\frac{\partial^2 \Lambda_{k_1,n_1}(r)}{\partial r} \Bigg|_{r=\rho(k_1/n_1)}\, =\, k_1 \left(\frac{F''(r)}{F(r)}-\left(\frac{F'(r)}{F(r)}\right)^2\right) + \frac{n_1}{r^2}\Bigg|_{r=\rho(k_1/n_1)}\,
=\, \frac{k_1 \mathcal{V}(\rho(k_1/n_1))}{\rho(k_1/n_1)^2}\, ,
$$
for
$$
\mathcal{V}(r)\, =\, r\, \frac{d}{dr} \left(\frac{r F'(r)}{F(r)}\right)\, =\, \sum_{n=1}^{\infty} \left(n - \mathcal{M}(r)\right)^2 \mu_n(r)\, .
$$
And if $\kappa \in (0,1)$ is in the analytic regime, then
$$
\frac{d\rho(\kappa)}{d\kappa}\, =\, -\frac{\rho(\kappa)}{\kappa^2 \mathcal{V}(\rho(\kappa))}\, =\, -\frac{\rho(\kappa) \mathcal{M}(\rho(\kappa))^2}{\mathcal{V}(\rho(\kappa))}\, .
$$
\end{lemma}

From the lemma, we conclude that, as long as $k/(n-m)$ is in the analytic regime, then
\begin{equation}
\label{eq:TaylorOne}
\begin{split}
 -\frac{1}{2}\, \frac{\partial^2 \Lambda_{k,n-m}(r)}{\partial r}\Bigg|_{r=\rho(k/(n-m))}
\left(\rho'\left(\frac{k}{n-m}\right)\right)^2\, \left(\Delta \kappa_{k,n,m}\right)^2\\
&\hspace{-3.5cm}=\, -\frac{\mathcal{M}(\rho(k/(n-m)))^2}{2 \mathcal{V}(\rho(k/(n-m)))}\,\cdot \frac{\big(n-m(k+1)\big)^2}{n^2k}\, .
\end{split}
\end{equation}
We have used $n-m = k \mathcal{M}(\rho(k/(n-m)))$.

\begin{conj}
Suppose that $\delta_1(\epsilon)>0$ for all $\epsilon>0$ and suppose that we have the condition
$$
\limsup_{r \to R^-} \frac{\mathcal{M}(r)^2}{\mathcal{V}(r)}\, <\, \infty\, .
$$
Then we have the Cram\'er  property for the Bell polynomials, as stated in Definition \ref{def:MainOne}.
\end{conj}

We call the condition
$$
\limsup_{r \to R^-} \frac{\mathcal{M}(r)^2}{\mathcal{V}(r)}\, <\, \infty\, ,
$$
the ``uniformly strict Cauchy-Schwarz'' condition, because it means
$$
\limsup_{r \to R^-}  \frac{\mathcal{M}(r)^2}{\mathcal{V}(r) +\mathcal{M}(r)^2}\, =\, 
\limsup_{r \to R^-} \frac{\left(\sum_{n=1}^{\infty} n \mu_n(r)\right)^2}{\sum_{n=1}^{\infty} n^2 \mu_n(r)}\, <\, 1\, .
$$
Of course the inequality with a non-strict inequality is trivial by Cauchy-Schwarz.
But we will need $\limsup_{r \to R^-} \mathcal{M}(r)^2/\mathcal{V}(r)<\infty$
in the rationale that we describe in the next section.
So we demand the uniformly strict Cauchy-Schwarz condition.

\section{Rationale for the conjecture: initial steps}

For the first consideration of the induction argument, let us assume that the condensation phase is empty.
We need to establish firstly that $\delta_k(\epsilon)>0$ for each $k \in \N$ and each $\epsilon>0$.
Then, after this first step, we want to establish that for each $\epsilon>0$ there is a $k_0 \in \N$ such that
$\delta_{k+1}(\epsilon) \geq \delta_k(\epsilon)$ for $k\geq k_0$.

Let us think of $k$ as fixed for the first part of the argument.
Clearly,
$$
\Big(\delta_k(\epsilon) > 0\Big)\, \Leftrightarrow\,
\Big(\liminf_{n \to \infty} 
\frac{Q_{k+1,n}}{\Theta_{k+1,n}}\, e^{\epsilon n} > 0 \Big)\, .
$$
Since
$$
-\Lambda_{k,n-m}\circ \rho(\kappa)\Bigg|_{k/(n-m)}^{(k+1)/n}\, \approx\, -\frac{1}{2}\, \frac{\partial^2 \Lambda_{k,n-m}(r)}{\partial r}\Bigg|_{r=\rho(k/(n-m))}
\left(\rho'\left(\frac{k}{n-m}\right)\right)^2\, \left(\Delta \kappa_{k,n,m}\right)^2\, ,
$$
using (\ref{eq:TaylorOne})
we can conclude that
$$
-\Lambda_{k,n-m}\circ \rho(\kappa)\Bigg|_{k/(n-m)}^{(k+1)/n}\, \geq\, - C_k\, ,
$$
as $n \to \infty$ for some constant $C_k$ (that depends on $\limsup_{r \to R^-} \mathcal{M}(r)^2/\mathcal{V}(r)<\infty$).
So then, using (\ref{eq:FiniteDifference}),
we find
$$
\frac{Q_{k+1,n}}{\Theta_{k,n+1}}\, e^{\epsilon n}\, 
\geq\, \frac{\delta_k(\epsilon)\, 
\sum_{m=1}^{n-k} Q_{1,m} \left(\rho\left(\frac{k+1}{n}\right)\right)^m e^{\epsilon m} e^{-C_k}}
{F\left(\rho\left(\frac{k+1}{n}\right)\right)}\, ,
$$
for $n$ sufficiently large.
But then we can write this as 
$$
\frac{Q_{k+1,n}}{\Theta_{k,n+1}}\, e^{\epsilon n}\, 
\geq\, \delta_k(\epsilon) e^{-C_k}\,  
\sum_{m=1}^{n-k} \mu_n\left(\rho\left(\frac{k+1}{n}\right)\right)\, .
$$
By Markov's inequality we can conclude the following.
\begin{lemma}
\label{lem:Markov}
For each $r \in (0,R)$, we have
$$
\sum_{m=m_1}^{\infty} \mu_n(r)\, \leq\, \frac{\mathcal{M}(r)}{m_1}\, ,
$$
for each $m_1 \in \N$ such that $m_1 \geq \mathcal{M}(r)$.
\end{lemma}
Hence, we have
$$
\frac{Q_{k+1,n}}{\Theta_{k,n+1}}\, e^{\epsilon n}\, 
\geq\, \delta_k(\epsilon) e^{-C_k}\,  \left(1 - \frac{n}{(k+1)(n-k)}\right)\, ,
$$
for $n$ sufficiently large.
But for $k$ fixed, if we take $n \to \infty$ this is at least $\delta_k(\epsilon) e^{-C_k}k/(k+1)>0$.

Of course, to make this rigorous, we would need to keep track of the remainder term in the Taylor expansion.
We do not wish to consider details at that level, here in this note.

\section{Rationale for the conjecture: induction step}

The second part of the argument is similar, but more complicated.
One may think of the induction argument as being somewhat akin to a differential
equation, with the finite difference of the induction step replacing the time evolution that is normally accomplished by
the derivative.
In this case we also have an underlying variable $\kappa \in (0,1)$.
So one can think of the recurrence relation
$$
Q_{k+1,n}\, =\, \sum_{m=1}^{n-k} Q_{1,m} Q_{k,n-m}\, ,
$$
as an integro-differential equation.
For the benefit of having a simple analogy, just imagine a partial differential equation.

We have two boundary points $\kappa=0$ and $\kappa=1$.
The technical difficulties at these two points are different.
Let us continue to assume that the condensation phase is empty.
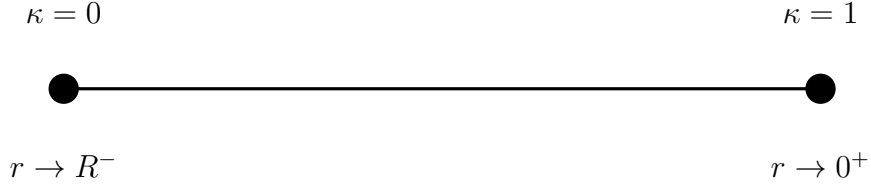
\begin{figure}
\begin{center}
\begin{tikzpicture}
\draw[very thick] (0,0) -- (10,0);
\fill (0,0) circle (2mm);
\fill (10,0) circle (2mm);
\draw (0,1) node[] {$\kappa=0$};
\draw (0,-1) node[] {$r\to R^-$};
\draw (10,1) node[] {$\kappa=1$};
\draw (10,-1) node[] {$r\to 0^+$};
\end{tikzpicture}
\end{center}
\caption{At the two boundary points there are different technical issues.
At $\kappa=0$, the issue is that we need the uniformly strict Cauchy-Schwarz condition.
At $\kappa=1$, the issue is that $n/(k+1)=1$ so $n-k=1$. This truncates the power series summation at one term.
But $\rho((k+1)/n)=0$, so $\boldsymbol{\mu}(0) = (1,0,0,\dots)$.
\label{fig:1}
}
\end{figure}
The basic idea of the induction argument is that for generic points, we can choose $k$ sufficiently large that
$$
e^{\epsilon m}
\exp\left(\left[ -\Lambda_{k,n-m}\circ \rho(\kappa)\right]\Bigg|_{k/(n-m)}^{(k+1)/n}\right)\, 
\geq\, 1\, .
$$
Even though the Taylor expansion of the difference in the exponent is negative, it is small in absolute value
(based in part on the uniformly strict Cauchy-Schwarz property).
If we assume $k$ is large enough, we may even assume
$$
e^{\epsilon m}
\exp\left(\left[ -\Lambda_{k,n-m}\circ \rho(\kappa)\right]\Bigg|_{k/(n-m)}^{(k+1)/n}\right)\, 
\geq\, e^{\epsilon/2}\, ,
$$
since all $m$ are at least $1$.
Then, by (\ref{eq:FiniteDifference}), we have 
$$
\frac{Q_{k+1,n}}{\Theta_{k+1,n}}\, e^{\epsilon n}\,
\geq\, e^{\epsilon/2} \delta_k(\epsilon)\, \frac{\sum_{m=1}^{n-k} Q_{1,m} \rho\left(\frac{k+1}{n}\right)^m}
{F\left(\rho\left(\frac{k+1}{n}\right)\right)}\, .
$$
So as long as we can prove that
$$
\sum_{m=1}^{n-k} Q_{1,m}  \rho\left(\frac{k+1}{n}\right)^m\, \geq\, e^{-\epsilon/2} F\left(\rho\left(\frac{k+1}{n}\right)\right)\, ,
$$
we will have established the induction step.
If the series included all terms, we would have equality without the need for $e^{-\epsilon/2}$:
$$
\sum_{m=1}^{\infty} Q_{1,m}  \rho\left(\frac{k+1}{n}\right)^m\, =\, F\left(\rho\left(\frac{k+1}{n}\right)\right)\, .
$$
But our basic argument for proving that the truncated series loses only a small relative error is Markov's inequality,
Lemma \ref{lem:Markov}.
This does not work when $n-k$ is ``small.''
But note that for $r$ small we have 
$$
F(r)\, \sim\, Q_{1,1} r\, ,\ \text{ as $r \to 0^+$.}
$$
Therefore,
$$
\lim_{\kappa \to 1} \mathcal{G}(\kappa)\, 
=\, \lim_{\kappa \to 1} \Big(\kappa \ln(F(\rho(\kappa))) - \ln(\rho(\kappa))\Big)\, =\, \ln(Q_{1,1})\, .
$$
This means we only need the first term in the sum at the exact boundary point.
Since $Q_{n,k+1}$ requires $n\geq k+1$ so that $n-k\geq 1$, we do have that one term, even in the truncated sum.
By careful consideration of this limit, we should be able to prove the inequality directly without using Markov's inequality.
For example, $r F'(r)/F(r) = 1 + r (Q_{1,2}/Q_{1,1}) + O(r^2)$ as $r\to 0^+$. 
This shows
$\rho(\kappa) \sim (1-\kappa)Q_{1,1}/Q_{1,2}$ as $\kappa \to 1^-$.

Choosing $\eta>0$ sufficiently small, we may be able to prove directly that for $n/(k+1) < 1+\eta$
the necessary induction step is true even though the sum is truncated at $m=n-k$.
Then for $k$ sufficiently large and $n= (1+x)(k+1)$, for $x>\eta$, the truncation occurs at $m_{\max}=1+x(k+1)$.
By Markov's inequality, we need $n/((k+1)m_{\max})$ to be small. But this is 
$$
\frac{n}{(k+1)m_{\max}}\, =\, \frac{(1+x)(k+1)}{(k+1)(1+x(k+1))}\, =\, \frac{1+x}{x(k+1)}\, =\, \frac{1}{k+1} + \frac{1}{x(k+1)}\,
\leq\, \frac{1}{k+1} + \frac{1}{\eta (k+1)}\, \leq\, \frac{2}{\eta(k+1)}\, .
$$
So as long as $k$ is chosen sufficiently large that  $2/\eta(k+1) < \epsilon/2$, the effect of the truncation should be negligible
for the purpose of establishing the induction argument.

\begin{remark}
We refrain from carrying out all details of the proof because the steps are becoming more complicated.
We hope that either some other researchers can finish the proof (possibly with the benefit of collective intelligence [also known as AI]
to carry out the tedious steps, for example).
Or else, we plan to come back to try to complete the proof before too long has passed, less than a year.
\end{remark}

\section{Proof simplification: the condensation phase}

If the condensation phase is not empty, then the argument becomes a bit easier, not more difficult.
Consider the desideratum
$$
e^{\epsilon m}
\exp\left(\left[ -\Lambda_{k,n-m}\circ \rho(\kappa)\right]\Bigg|_{k/(n-m)}^{(k+1)/n}\right)\, 
\geq\, 1\, .
$$
In the condensation phase $\rho(\kappa)$ has become frozen at $R$.
So, depending on $n$, $k$ and $m$, the $\rho$ at the two values are equal.
Thus instead of Taylor expanding, the difference is just zero, directly.
We also note that the derivative is continuous at the critical point.
It is just the second derivative that drops to $0$ in the condensation phase.
Therefore, since the Taylor expansion that we have considered is to 2nd order,
we do have that the 2nd derivative is $0$ in the condensation phase.
So the effect of the condensation phase is to make terms that we are trying to bound
to become even smaller (in fact $0$).

\section{Examples and known results}

The first four examples are the first four examples that Comtet gives in his monograph.
\begin{itemize}
\item[(1)] If you let $w_n=1$ for all $n \in \N$, then $B_{n,k} = S_{n,k}$ is the Stirling number of the second kind.
$$
\begin{tikzpicture}
\draw (-0.25,-0.25) -- (5.5,-0.25);
\draw (0.5,0.25) -- (0.5,-2.75);
\draw (0,0) node[] {$n$};
\draw (1,0) node[] {$S_{n,1}$};
\draw (2,0) node[] {$S_{n,2}$};
\draw (3,0) node[] {$S_{n,3}$};
\draw (4,0) node[] {$S_{n,4}$};
\draw (5,0) node[] {$S_{n,5}$};
\draw (0,-0.5) node[] {$1$};
\draw (1,-0.5) node[] {$1$};
\draw (0,-1) node[] {$2$};
\draw (1,-1) node[] {$1$};
\draw (2,-1) node[] {$1$};
\draw (0,-1.5) node[] {$3$};
\draw (1,-1.5) node[] {$1$};
\draw (2,-1.5) node[] {$3$};
\draw (3,-1.5) node[] {$1$};
\draw (0,-2) node[] {$4$};
\draw (1,-2) node[] {$1$};
\draw (2,-2) node[] {$7$};
\draw (3,-2) node[] {$6$};
\draw (4,-2) node[] {$1$};
\draw (0,-2.5) node[] {$5$};
\draw (1,-2.5) node[] {$1$};
\draw (2,-2.5) node[] {$15$};
\draw (3,-2.5) node[] {$25$};
\draw (4,-2.5) node[] {$10$};
\draw (5,-2.5) node[] {$1$};
\end{tikzpicture}
$$
More information is available at 
$$
\text{\url{https://oeis.org/A008277}}
$$
\item[(2)] If you let $w_n=n!$ for all $n \in \N$, then $B_{n,k}$ equals $\binom{n-1}{k-1}\, \frac{n!}{k!}$
which Comtet called the Lah numbers.
\item[(3)] If you let $w_n=(n-1)!$ for all $n \in \N$, then $B_{n,k}=|s_{n,k}|$ the unsigned Stirling numbers
of the first kind.
$$
\begin{tikzpicture}
\draw (-0.25,-0.25) -- (5.5,-0.25);
\draw (0.5,0.25) -- (0.5,-2.75);
\draw (0,0) node[] {$n$};
\draw (1,0) node[] {$|s_{n,1}|$};
\draw (2,0) node[] {$|s_{n,2}|$};
\draw (3,0) node[] {$|s_{n,3}|$};
\draw (4,0) node[] {$|s_{n,4}|$};
\draw (5,0) node[] {$|s_{n,5}|$};
\draw (0,-0.5) node[] {$1$};
\draw (1,-0.5) node[] {$1$};
\draw (0,-1) node[] {$2$};
\draw (1,-1) node[] {$1$};
\draw (2,-1) node[] {$1$};
\draw (0,-1.5) node[] {$3$};
\draw (1,-1.5) node[] {$2$};
\draw (2,-1.5) node[] {$3$};
\draw (3,-1.5) node[] {$1$};
\draw (0,-2) node[] {$4$};
\draw (1,-2) node[] {$6$};
\draw (2,-2) node[] {$11$};
\draw (3,-2) node[] {$6$};
\draw (4,-2) node[] {$1$};
\draw (0,-2.5) node[] {$5$};
\draw (1,-2.5) node[] {$24$};
\draw (2,-2.5) node[] {$50$};
\draw (3,-2.5) node[] {$35$};
\draw (4,-2.5) node[] {$10$};
\draw (5,-2.5) node[] {$1$};
\end{tikzpicture}
$$
More information is available at 
$$
\text{\url{https://oeis.org/A008275}}
$$
\item[(4)] If we let $w_n = n$ for all $n \in \N$ then $B_{n,k}$ equals $\binom{n}{k} k^{k-1}$, which Comtet
calls the idempotent numbers.
More information is available at 
$$
\text{\url{https://oeis.org/A059297}}
$$
\end{itemize}

Another example that Comtet gives later in the book is this one:
\begin{itemize}
\item[(5)] For each $n \in \N$, let $\sigma(n) = \sum_{k | n} k$ the sum of divisors. 
Then let $w_n  = (n-1)! \sigma(n)$ for each $n$. This is example 10 on page 159 of Comtet's monograph.
Then $B_{n,k}$ are the d'Arcais numbers $A_{n,k}$.
$$
\begin{tikzpicture}
\draw (-0.25,-0.25) -- (5.5,-0.25);
\draw (0.5,0.25) -- (0.5,-2.75);
\draw (0,0) node[] {$n$};
\draw (1,0) node[] {$A_{n,1}$};
\draw (2,0) node[] {$A_{n,2}$};
\draw (3,0) node[] {$A_{n,3}$};
\draw (4,0) node[] {$A_{n,4}$};
\draw (5,0) node[] {$A_{n,5}$};
\draw (0,-0.5) node[] {$1$};
\draw (1,-0.5) node[] {$1$};
\draw (0,-1) node[] {$2$};
\draw (1,-1) node[] {$3$};
\draw (2,-1) node[] {$1$};
\draw (0,-1.5) node[] {$3$};
\draw (1,-1.5) node[] {$8$};
\draw (2,-1.5) node[] {$9$};
\draw (3,-1.5) node[] {$1$};
\draw (0,-2) node[] {$4$};
\draw (1,-2) node[] {$42$};
\draw (2,-2) node[] {$59$};
\draw (3,-2) node[] {$18$};
\draw (4,-2) node[] {$1$};
\draw (0,-2.5) node[] {$5$};
\draw (1,-2.5) node[] {$144$};
\draw (2,-2.5) node[] {$450$};
\draw (3,-2.5) node[] {$215$};
\draw (4,-2.5) node[] {$30$};
\draw (5,-2.5) node[] {$1$};
\end{tikzpicture}
$$
More information is available at 
$$
\text{\url{https://oeis.org/A008298}}
$$
\end{itemize}
It is the d'Arcais numbers which come closest to our main interest.
Abdelmalek Abdesselam considered commuting $\ell$-tuples of permutations $(\pi_1,\dots,\pi_{\ell})$
with $\pi_1,\dots,\pi_{\ell} \in S_n$.
Then considering the group generated by these elements $\langle \pi_1,\dots,\pi_{\ell} \rangle$,
one can consider the number of orbits of this subgroup acting on $[n] = \{1,\dots,n\}$.
If one requires the orbit number to be $k$, then the number of such $\ell$-tuples 
are written as $A(\ell,n,k)$.
In \cite{Abdesselam1}, Abdesselam, Brunialti, Doan and Velie showed that
$$
A(\ell,n,k)\, =\, \frac{n!}{k!}\, \sum_{\substack{(\nu(1),\dots,\nu(k)) \in \N^k\\ \nu(1)+\dots+\nu(k)=n}} \prod_{j=1}^{k}
\frac{\mathfrak{B}(\ell,\nu(j))}{\nu(j)}\, ,
$$
for $\mathfrak{B}(\ell,n)$ being the multiplicative function (in the number theory sense that 
$\mathfrak{B}(\ell,ab) = \mathfrak{B}(\ell,a) \mathfrak{B}(\ell,b)$ when $\operatorname{gcd}(a,b)=1$) such that
$$
\mathfrak{B}(\ell,q^m)\, =\, \frac{(q^{\ell}-1)(q^{\ell+1}-1)\cdots (q^{\ell+m-1}-1)}{(q-1)(q^2-1)\cdots (q^m-1)}\, ,
$$
when $m\geq 0$ and $q$ is prime.
This result had previously been proved by Bryan and Fulman using generating functions (and the wreath product on $S_n$) 
in a celebrated result \cite{BryanFulman}.
But Abdesselam, Brunialti, Doan and Velie gave a direct bijective proof using discrete tori with twists.
From the Bryan and Fulman article,
$$
\sum_{n=1}^{\infty} \frac{\mathfrak{B}(\ell,n)}{n}\, z^n\, =\, 
-\sum_{\delta_1,\dots,\delta_{\ell-1} \in \N} \delta_1^{\ell-2} \delta_2^{\ell-3} \cdots \delta_{\ell-2}\, \ln\left(1-z^{\delta_1 \cdots \delta_{\ell-1}}\right)\, .
$$
In particular, for $\ell=2$ this is the example of the d'Arcais numbers.

We note that for each $\ell \in \{2,3,\dots\}$ these are Bell polynomials and the $F(z)$ is as above, for $w_n = (n-1)! B(\ell,n)$.
It is easy to see that
$$
\sum_{n=1}^{\infty} \frac{\mathfrak{B}(\ell,n)}{n}\, r^n\, \sim\, \frac{(\ell-2)! \zeta(2)\cdots \zeta(\ell)}{(1-r)^{\ell-1}}\, ,\ \text{ as $r \to 1^-$.}
$$
This means that the generating function has radius of convergence $R=1$ and does satisfy the uniformly strict Cauchy-Schwarz inequality.
Therefore, our conjecture is that the large deviation rate function is given by Cram\'er's formula,
and the proof can be carried out using Erd\H{o}s's elementary method.

We do state that for these particular examples, a central limit theorem was already proved in \cite{AbdesselamStarr}
but that was not as strong as the large deviation principle is.
In particular, if one can obtain the large deviation principle, including the Bahadur-Rao term, then 
one may be able to re-derive a result of Abdesselam \cite{Abdesselam2}, where he obtained precise asymptotics
for the numbers $A(\ell,n,k)$ using the Mellin transform and a method pioneered by Bringmann, Franke and Heim in \cite{BFH}.
Indeed, for $\ell=2$, one of the authors did do this (using the circle method instead of Erd\H{o}s's elementary method).
But for that case, the modular symmetry of the generating function was used, since the generating function in that case is 
essentially Dedekind's $\eta$-function.
If one wanted to try to reproduce their results for all $\ell\geq 3$ then one would need to show that the circle method
applies to those cases (which is probable) and also find a generalization of modularity for those generating functions,
which we are not at all sure about ourselves (since we are not number theorists).

\subsection{One counterexample}

Essentially, the uniformly strict Cauchy-Schwarz property applies to all of the examples except the Stirling numbers of the second kind.
We mention that the Stirling numbers of the second kind $S(n,k)$ are the most famous and useful example of Bell polynomials.
But since $w_n\equiv 1$, they are also the example that most clearly fails the ``uniformly strict Cauchy-Schwarz'' property.

\subsection{Previous results for the Stirling numbers}

The textbook by R.~B.~Dingle \cite{Dingle} uses contour integrals to motivate results for asymptotics.
Two examples are the Stirling numbers of the first and second kind.
But the rigorous proofs are due to Moser and Wyman \cite{MoserWyman1,MoserWyman2} and Temme \cite{Temme}.

Firstly, there is the elementary formula (Cauchy product)
$$
F(z)^k\, =\, \sum_{n=k}^{\infty} B_{n,k}\, \frac{k!}{n!}\, z^n\, .
$$
Using this and Cauchy's integral formula, one may write
$$
B_{n,k}\, \frac{k!}{n!}\, =\, \int_{\mathcal{C}(0;r)} \frac{F(z)^k}{z^n}\, \cdot \frac{dz}{2\pi i z}\, ,
$$
for any $r \in (0,R)$. Then one can try to apply the Hayman saddle point method,
as described by Flajolet and Sedgewick \cite{FlajoletSedgewick}.
We note that the choice of $r$ is to minimize $k \ln(F(r;\boldsymbol{w})) - n \ln(r)$.

\begin{itemize}
\item[(A)] For example (1) above, we have
$$
F(z)\, =\, \sum_{n=1}^{\infty} \frac{1}{n!}\, z^n\, =\, e^z-1\, .
$$
The radius of convergence is $R=\infty$.
For any number $t \in (0,\infty)$, let $\kappa \in (0,1)$ be defined such that
$$
\kappa\, =\, \frac{1-e^{-t}}{t}\, .
$$
Then if we take a sequence $k_n$ such that $\lim_{n \to \infty} k_n/n=\kappa$, then
the Stirling numbers of the second kind satisfy
$$
\lim_{n \to \infty} \frac{1}{n}\, \ln\left(S(n,k_n)\, \frac{k_n!}{n!}\right)\, =\, \ln(t) - \kappa \ln(e^{t}-1)\, .
$$
This may be read from Example 4 on Page 199 of R.~B.~Dingle's monograph \cite{Dingle}.

We note that while our ``uniformly strict Cauchy-Schwarz'' property is not satisfied for this model, by other methods, it is known
that the large deviation rate function is given by the Cram\'er formula.
We view it as an interesting challenge to try to adapt the elementary method to handle this example, possibly
by keeping better track of the large deviation rate function approximation $\Theta_{k,n}$ for $n \to \infty$.
\item[(B)] For example (3) above, we have
$$
F(z)\, =\, \sum_{n=1}^{\infty} \frac{1}{n}\, z^n\, =\, -\ln(1-z)\, .
$$
The radius of convergence is $R=1$.
Since
$$
\mathcal{M}(z)\, =\, \frac{z F'(z)}{F(z)}\, =\, -\frac{z}{(1-z)\ln(1-z)}\, ,
$$
and
$$
\mathcal{V}(z)\, =\, z\, \frac{d}{dz}\left(\frac{z F'(z)}{F(z)}\right)\,
=\, \left(-\frac{\ln(1-z)}{z}-1\right) \mathcal{M}(z)^2\, ,
$$
that means that this example does satisfy the ``uniformly strict Cauchy-Schwarz'' property.

For any number $t \in (0,\infty)$, let $\kappa \in (0,1)$ be defined such that
$$
\kappa\, =\, \frac{t}{e^t-1}\, .
$$
Then if we take a sequence $k_n$ such that $\lim_{n \to \infty} k_n/n=\kappa$, 
the unsigned Stirling numbers of the first kind satisfy
$$
\lim_{n \to \infty} \frac{1}{n}\, \ln\left(|s(n,k_n)|\, \frac{k_n!}{n!}\right)\, =\, \kappa \ln(t) - \ln(1-e^{-t})\, .
$$
This may be read from Example 5 on pp.~199--200 of Dingle's monograph.
\end{itemize}

For these two examples, there are also elementary techniques that bypass contour integration.
There are bounds due to Arratia and DeSalvo using the Chen-Stein method \cite{ArratiaDeSalvo}.
There is also recent work on this topic using elementary techniques analogous to the Erd\"os method by 
Hwang, Li and Zacharovas
\cite{HwangLiZacharovas}.
We are actively reading their article now, to see how to bypass our ``uniformly strict Cauchy-Schwarz'' property.

\section{Relevance of log concavity}

Because $\mathcal{G}$ is concave, we see that whenever Cram\'er's formula holds, we have that
$\ln(Q_{k,n})/n$ is asymptotically concave in $k/n$. Heim and Neuhauser had conjectured that log-concavity holds
for $A(\ell,n,k)$ for $\ell=2$, in \cite{HeimNeuhauser}. In \cite{StarrSmall} and more recently in the elegant explicit example from 
Charlton, Heim and Stumpenhusen \cite{CHS}, small counterexamples were found.
But the asymptotic log concavity is true, (by the Bahadur-Rao large deviation principle in \cite{StarrLDP}).
The conjecture, if it were proved true, would show generally why asymptotic log-concavity should hold.
It would also imply asymptotic log-concavity for $\ell\geq 3$, which is already known to be true due to Abdesselam's 
{\em tour-de-force} calculation \cite{Abdesselam2}. But the proof would be more elementary.
In \cite{Abdesselam0} Abdesselam has generalized the original conjecture of Heim and Neuhauser for $\ell\geq 3$.
As of now, we do not know if any research has been done on whether there may exist counterexamples for small values.
The validity of Abdesselam's conjecture seems like a fascinating open problem!

\section{One example for condensation}

The Online Encyclopedia of Integer Sequences is a great trove of information for particular combinatorial sequences.
Consider 
$$
w_n\, =\, n^{n-2}\, ,
$$
for $n \in \N$. This is well-known. It is the number of labeled trees on $n$ vertices. See for example
$$
\text{\url{https://en.wikipedia.org/wiki/Double_counting_(proof_technique)}}
$$
under the subsection, ``Counting trees.''
On OEIS, it is stated that $B_{n,k}$ is the number of forests with $n$ nodes and $k$ labeled trees.
That seems reasonable:
$$
\text{\url{https://oeis.org/A105599}}\, .
$$
We note that
$$
Q_{1,n}\, =\, \frac{n^{n-2}}{n!}\, \sim\, \frac{e^n}{n^{5/2}\, \sqrt{2\pi}}\, .
$$
So this has $R=e^{-1}$. Also using the Lambert $W$-function based at $0$,
$$
W_0(z)\, =\, \sum_{n=1}^{\infty} \frac{(-n)^{n-1}}{n!} z^n\, ,
$$
we see that $\int_0^z \xi^{-1} W_0(\xi)\, d\xi = -F(-z)$. So using the tree function $T(z) =-W_0(-z)$, we have
$F(z) =\int_0^z \xi^{-1} T(\xi)\, d\xi$.
Using $T(x) e^{-T(x)} = x$ we have $\xi^{-1} =e^t/r$ for $t=T(\xi)$.
And $d\xi = T'(\xi) e^{-T(\xi)} (1-T(\xi))\, d\xi$.
So $F(z) = \int_0^{T(z)} (1-t)\, dt = T(z) - \frac{1}{2}\, T(z)^2$.
Then using $W(-e^{-1})=-1$ this implies $\lim_{r \uparrow e^{-1}} F(r) = 1/2$.
See
$$
\text{\url{https://en.wikipedia.org/wiki/Lambert_W_function}}
$$
under the subsection, ``Special values.''
Moreover, $r F'(r) = T(r)$ so that $r F'(r)/F(r) = T(r)/(T(r)-\frac{1}{2}T(r)^2)=1/(1-\frac{1}{2}T(r))$.
Thus $\lim_{r\uparrow e^{-1}} r F'(r)/F(r) = 2$.
This means that there is a condensation phase in $(0,1/2)$.

\begin{remark}
We did not yet establish the conjecture. So our work does not establish this property for this example, rigorously, yet.
But it is extremely well-known in the literature.
See for example {\L}uczak and Pittel \cite{LuczakPittel}.
Or also see Aldous and Pitman \cite{AldousPitman}.
\end{remark}

\section*{Acknowledgment}

S.S.~ benefited from many conversations with Google Gemini, especially to find relevant literature.
In particular, the example for condensation was found in collaboration with Google Gemini: Google Gemini
alerted us to its existence in OEIS. We checked all of Google Gemini's calculuations using Wikipedia.
Following some suggestions, we refer to this as ``collective intelligence,'' because not only is the intelligence
of Google Gemini and similar machines ``artificial,'' but for our purposes it is more relevant that they have aggregated
the known research of all of our colleagues in the research field.
We owe a debt to Google Gemini, but even more we owe a debt to all other researchers whose work we use and we 
contribute towards.

\appendix

\section{The Cram\'er argument for the distribution function}

The main point of the Erd\H{o}s elementary argument is a kind of Tauberian theorem.
In probability theory, this would be called a ``local limit theorem,'' although a true local limit theorem
in this context would include a Bahadur-Rao correction, which we do not want to consider here.

In this section, we want to describe the standard Cram\'er argument, which leads to a weak limit theorem, instead.
For each $k$, let us denote the probability mass function on $n \in \{k,k+1,\dots\}$ as 
$$
\mu_{n}^{(k)}(r)\, =\, \frac{Q_{k,n} r^k}{F(r)^k}\, .
$$
(So $\boldsymbol{\mu}(r) = \boldsymbol{\mu}^{(1)}(r)$.)
Let the distribution function be denoted $\mathcal{D}^{(k)}(\cdot,r)$ where
$$
\mathcal{D}^{(k)}(x,r)\, =\, \sum_{n=k}^{\infty} \mathbf{1}_{(-\infty,x]}(k) \mu_{n}^{(k)}(r)\, .
$$
Then, by Chebyshev's inequality, for any $t>0$
$$
1-\mathcal{D}^{(k)}(x,r)\, \leq\, \frac{F(re^t)^k}{e^{xt} F(r)^k}\, ,
$$
while for $t>0$, we have
$$
\mathcal{D}^{(k)}(x,r)\, \leq\, \frac{F(re^{-t})^k}{e^{-xt} F(r)^k}\, .
$$

Therefore, we conclude the following.
\begin{lemma}
Fix $r \in (0,R)$.
\begin{itemize}
\item
Suppose $s \in (0,r)$. Then
$$
\frac{1}{k}\, \ln\left(\mathcal{D}^{(k)}(k\mathcal{M}(s),r)\right)\, \leq\, \ln\left(\frac{F(s)}{F(r)}\right) - \mathcal{M}(s)\, \ln\left(\frac{s}{r}\right)\, .
$$
\item
Suppose $s \in (r,R)$. Then
$$
\frac{1}{k}\, \ln\left(1-\mathcal{D}^{(k)}(k\mathcal{M}(s),r)\right)\, \leq\, \ln\left(\frac{F(s)}{F(r)}\right) - \mathcal{M}(s)\, \ln\left(\frac{s}{r}\right)\, .
$$
\item In case $\mathcal{M}(R)<\infty$, if $x>k \mathcal{M}(R) $ then
$$
\frac{1}{k}\, \ln\left(1-\mathcal{D}^{(k)}(x,r)\right)\, \leq\, \ln\left(\frac{F(R)}{F(r)}\right) - \frac{x}{k}\, \ln\left(\frac{s}{r}\right)\, .
$$
\end{itemize}
\end{lemma}

This is intuitively quite close to (\ref{ineq:Upper}).
But the difference is that now we are bounding the cumulative distribution function instead of the probability mass function.
This gives enough continuity to run Cram\'er's formula, as usual.
In other words, we can get reversed bounds.

It is easy to see that
$$
\sum_{n=k}^{\infty} n \mu^{(k)}_n(r)\, =\, r\, \frac{d}{dr}\, \left(\ln \left(F(r)^k\right)\right)\, =\, k \mathcal{M}(r)\, .
$$
And similarly,
$$
\sum_{n=k}^{\infty} \left(n - k \mathcal{M}(r)\right)^2 \mu^{(k)}_n(r)\,
=\, \left(r\, \frac{d}{dr}\right)^2\left(\ln \left(F(r)^k\right)\right)\, =\, k \mathcal{V}(r)\, .
$$
Therefore, by the Chebyshev inequality,
$$
\mathcal{D}^{(k)}\left(k\mathcal{M}(r)+k^{1/2} a\, ;\ r\right)
-\mathcal{D}^{(k)}\left(k\mathcal{M}(r)-k^{1/2} a\, ;\ r\right)\, \geq\, 1 - \frac{\mathcal{V}(r)}{a^2}\, .
$$
In particular, we have
$$
\mathcal{D}^{(k)}\left(k\mathcal{M}(r)+\sqrt{2k\mathcal{V}(r)}\, ;\ r\right)
-\mathcal{D}^{(k)}\left(k\mathcal{M}(r)-\sqrt{2k\mathcal{V}(r)}\, ;\ r\right)\, \geq\, \frac{1}{2}\, .
$$
So for any $s<r$, we have, by the exponential property of the measures,
$$
\mathcal{D}^{(k)}\left(k\mathcal{M}(r)+\sqrt{2k\mathcal{V}(r)}\, ;\, s\right)
-\mathcal{D}^{(k)}\left(k\mathcal{M}(r)-\sqrt{2k\mathcal{V}(r)}\, ;\, s\right)\, 
\geq\, \frac{1}{2}\, \left(\frac{s}{r}\right)^{k\mathcal{M}(r)+\sqrt{2k\mathcal{V}(r)}}\, \frac{F(r)^k}{F(s)^k}\, .
$$
In particular,
$$
1-\mathcal{D}^{(k)}\left(k\mathcal{M}(r)-\sqrt{2k\mathcal{V}(r)}\, ;\, s\right)\, 
\geq\, \frac{1}{2}\, \left(\frac{s}{r}\right)^{k\mathcal{M}(r)+\sqrt{2k\mathcal{V}(r)}}\, \frac{F(r)^k}{F(s)^k}\, .
$$
Now given $s \in (0,R)$, choose $r \in (s,R)$.
Let $q_k \in (r,R)$ be chosen such that 
$k\mathcal{M}(r)=k\mathcal{M}(q_k)-\sqrt{2k\mathcal{V}(q_k)}$ for each $k\in \N$.
Of course by continuity of $\mathcal{M}$ and $\mathcal{V}$ (in fact analyticity) we know $q_k \to r$.
Then we have
$$
1-\mathcal{D}^{(k)}(k\mathcal{M}(r)\, ;\, s)\, \geq\, 
\frac{1}{2}\, \left(\frac{s}{q_k}\right)^{k\mathcal{M}(q_k)+\sqrt{2k\mathcal{V}(q_k)}}\, \frac{F(q_k)^k}{F(s)^k}\, .
$$
Taking the logarithm and taking the limit we see
$$
\liminf_{k \to \infty} \frac{1}{k}\, \ln\left(1-\mathcal{D}^{(k)}(k\mathcal{M}(r)\, ;\, s)\right)
\geq\, \mathcal{M}(r)\, \ln\left(\frac{s}{r}\right) + \ln\left(\frac{F(r)}{F(s)}\right)\, .
$$
Note that we chose our variables for this half of the argument so that the functions of $r$ and $s$
have reversed.
The other bound works symmetrically.
But we cannot go beyond $R$. So we conclude the following.

\begin{theorem}
Fix $r \in (0,R)$.
\begin{itemize}
\item
Suppose $s \in (0,r)$. Then
$$
\frac{1}{k}\, \ln\left(\mathcal{D}^{(k)}(k\mathcal{M}(s),r)\right)\, =\, \ln\left(\frac{F(s)}{F(r)}\right) - \mathcal{M}(s)\, \ln\left(\frac{s}{r}\right)\, .
$$
\item
Suppose $s \in (r,R)$. Then
$$
\frac{1}{k}\, \ln\left(1-\mathcal{D}^{(k)}(k\mathcal{M}(s),r)\right)\, =\, \ln\left(\frac{F(s)}{F(r)}\right) - \mathcal{M}(s)\, \ln\left(\frac{s}{r}\right)\, .
$$
\end{itemize}
\end{theorem}

\baselineskip=12pt
\bibliographystyle{plain}

\end{document}